\documentclass[a4paper,11pt]{article}
\usepackage[T1]{fontenc}
\usepackage[utf8]{inputenc}
\usepackage[english]{babel}
\usepackage[margin=0.8in]{geometry}
\usepackage{microtype}
\usepackage{braket}
\usepackage{amsmath}
\usepackage{amssymb}
\usepackage{amsthm, aliascnt}
\usepackage{mathtools}
\usepackage{mathrsfs}
\usepackage{xcolor}
\usepackage[hyphens]{url}
\usepackage{hyperref}
\usepackage{enumitem}
\usepackage{tikz}
\usepackage[maxbibnames=99]{biblatex}
\usepackage{csquotes}
\usetikzlibrary{patterns}
\usepackage{tikz}
\usepackage{esint}
\usepackage{mleftright}

\newcommand{\Om}{\Omega}

\newcommand{\la}{\lambda}

\newcommand{\La}{\Lambda}

\newcommand{\De}{\Delta}
\newcommand{\na}{\nabla}

\newcommand{\Ga}{\Ga}
\newcommand{\vp}{\varphi}

\newcommand{\fr}[2]{\frac{#1}{#2}}

\newcommand{\sbs}{\subset}
\newcommand{\sbse}{\subseteq}
\newcommand{\Spn}{\mathbb{S}^{n-1}}

\newcommand{\ml}{\mleft}
\newcommand{\mr}{\mright}

\newcommand{\ddfr}[2]{\frac{\displaystyle #1}{\displaystyle #2}}

\usepackage{xpatch}
\makeatletter
\newcounter{proofpart}
\xpretocmd{\proof}{\setcounter{proofpart}{0}}{}{}
\newcommand{\proofpart}[1]{%
  \par
  \addvspace{\medskipamount}%
  \stepcounter{proofpart}%
  \noindent\emph{Step \theproofpart: #1}\par\nobreak\smallskip
  \@afterheading
}
\makeatother

\usepackage{import}

\newcommand{\incsvg}[2][1]{%
\def\svgwidth{#1\columnwidth}
\begingroup%
  \makeatletter%
  \providecommand\color[2][]{%
    \errmessage{(Inkscape) Color is used for the text in Inkscape, but the package 'color.sty' is not loaded}%
    \renewcommand\color[2][]{}%
  }%
  \providecommand\transparent[1]{%
    \errmessage{(Inkscape) Transparency is used (non-zero) for the text in Inkscape, but the package 'transparent.sty' is not loaded}%
    \renewcommand\transparent[1]{}%
  }%
  \providecommand\rotatebox[2]{#2}%
  \newcommand*\fsize{\dimexpr\f@size pt\relax}%
  \newcommand*\lineheight[1]{\fontsize{\fsize}{#1\fsize}\selectfont}%
  \ifx\svgwidth\undefined%
    \setlength{\unitlength}{492.13970515bp}%
    \ifx\svgscale\undefined%
      \relax%
    \else%
      \setlength{\unitlength}{\unitlength * \real{\svgscale}}%
    \fi%
  \else%
    \setlength{\unitlength}{\svgwidth}%
  \fi%
  \global\let\svgwidth\undefined%
  \global\let\svgscale\undefined%
  \makeatother%
  \begin{picture}(1,0.44422203)%
    \lineheight{1}%
    \setlength\tabcolsep{0pt}%
    \put(0,0){\includegraphics[width=\unitlength,page=1]{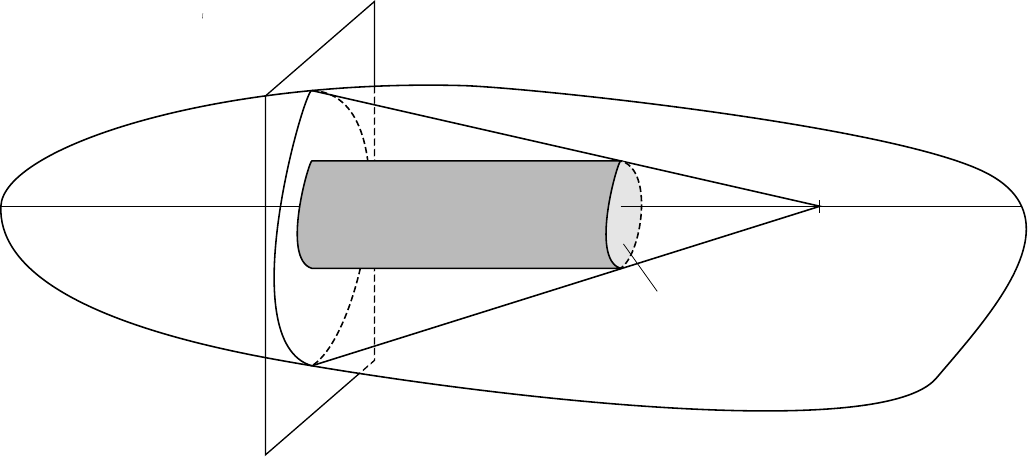}}%
  \end{picture}%
\endgroup%

}

\hypersetup{
                colorlinks=true,
                linkcolor={magenta},
                citecolor={cyan},
                breaklinks=true}

\theoremstyle{plain}
\newtheorem{teor}{Theorem}
\numberwithin{teor}{section}
\numberwithin{equation}{section}

\theoremstyle{definition}
\newaliascnt{defi}{teor}
\newtheorem{defi}[defi]{Definition}
\aliascntresetthe{defi}

\theoremstyle{plain}
\newaliascnt{lemma}{teor}
\newtheorem{lemma}[lemma]{Lemma}
\aliascntresetthe{lemma}
\theoremstyle{plain}
\newaliascnt{prop}{teor}
\newtheorem{prop}[prop]{Proposition}
\aliascntresetthe{prop}
\theoremstyle{plain}
\newaliascnt{conjecture}{teor}

\aliascntresetthe{conjecture}
\theoremstyle{plain}
\newaliascnt{cor}{teor}

\aliascntresetthe{cor}
\theoremstyle{definition}
\newaliascnt{ex}{teor}

\aliascntresetthe{ex}
\theoremstyle{definition}
\newaliascnt{oss}{teor}
\newtheorem{oss}[oss]{Remark}
\aliascntresetthe{oss}
\theoremstyle{plain}
\newtheorem{open}{Open problem}
\DeclarePairedDelimiter{\abs}{\lvert}{\rvert}

\newcommand{\R}{\mathbb{R}}

\DeclareMathOperator{\divv}{div}

\DeclareMathOperator{\supp}{supp}

\title{An improved  version of a spectral inequality by Payne}
\author{Paolo Acampora, Emanuele Cristoforoni, Carlo Nitsch, Cristina Trombetti}
\date{}

\newcommand{\Addresses}{{
 \bigskip 
 \footnotesize 
 
 \textsc{Dipartimento di Matematica e Applicazioni ``R. Caccioppoli'', Universit\`a degli studi di Napoli Federico II, Via Cintia, Complesso Universitario Monte S. Angelo, 80126 Napoli, Italy.}\par\nopagebreak 
 
 \medskip 
 
 \textit{E-mail address}, P.~Acampora: \texttt{paolo.acampora@unina.it} 

 \medskip 
 
 \textit{E-mail address}, C.~Nitsch: \texttt{c.nitsch@unina.it}

  \medskip 
 
 \textit{E-mail address}, C.~Trombetti: \texttt{cristina@unina.it} 
 
 \medskip 
 
\textsc{Mathematical and Physical Sciences for Advanced Materials and Technologies, Scuola Superiore Meridionale, Largo San Marcellino 10, 80138, Napoli, Italy.}\par\nopagebreak 
 
 \medskip 
 
 \textit{E-mail address}, E.~Cristoforoni: \texttt{emanuele.cristoforoni@unina.it} 
}}

\begin{document}

\maketitle

\begin{abstract}
    A celebrated inequality by Payne relates the first eigenvalue of the Dirichlet Laplacian to the first eigenvalue of the buckling problem. Motivated by the goal of establishing a quantitative version of this inequality, we show that Payne’s original estimate—which is not sharp—can in fact be improved. Our result provides a refined spectral bound and opens the way to further investigations into quantitative enhancements of classical inequalities in spectral theory.
    

\end{abstract}

\section{Introduction}
Let $\Omega\subset\R^n$ be a bounded convex set and let $\lambda(\Omega)$ and $\Lambda(\Omega)$ be the principal eigenvalues of the Dirichlet eigenvalue problem 
\[\begin{cases}
    -\Delta u = \lambda u &\text{in }\Omega,\\[5pt]
    u=0 &\text{on }\partial \Omega,
\end{cases}\]
and of the Buckling eigenvalue problem
\[\begin{cases}
    -\Delta^2 v = \Lambda \Delta v &\text{in }\Omega,\\[5pt]
    v=\partial_\nu v = 0 &\text{on }\partial \Omega,
\end{cases}\]
where $\partial_\nu v$ denotes the normal derivative on the boundary of $\Om$, respectively. The two eigenvalues can be variationally characterized as
\[\lambda(\Omega)=\min_{u\in H^1_0(\Omega)} \dfrac{\int_\Omega \abs{\nabla u}^2\,dx}{\int_\Omega u^2\,dx }\quad \text{ and }\quad 
\Lambda(\Omega)=\min_{v\in H^2_0(\Omega)} \dfrac{\int_\Omega (\Delta v)^2\,dx}{\int_\Omega \abs{\nabla v}^2\,dx }.\]
Payne in \cite{Payne1955} (see also \cite{Payne}) proved that
\begin{equation}\label{Payneineq}\Lambda(\Omega)\le4\lambda(\Omega),\end{equation}
moreover, it was thought that the infinite strip asymptotically achieved the equality sign, and thus, the factor $4$ could not be improved (see for instance \cite{Payne1955} or \cite[Remark 8.13 (iv)]{AshGesztMitrShtTreschl2013}). However, in the present paper, we will prove the following.
\begin{teor}\label{main}  Let $n\ge2$. There exists $C_n<4$ such that for every convex set $\Om\subseteq\R^n$
  \[
\La(\Om)\le C_n \la(\Om).
  \]
\end{teor}
In the case $n=1$, inequality \eqref{Payneineq} always holds with the equality sign. Indeed if $\Om$ is the interval $(0,L)$, as it it well known
\[u(x)=\sin\left(\dfrac{\pi}{L}x\right),\quad\text{and}\quad\lambda(\Om)=\dfrac{\pi^2}{L^2},\]
while, by explicit computations, one can prove that
\[v(x)=\dfrac{1}{2}\left(1-\cos\left(2\dfrac{\pi}{L}x\right)\right),\quad\text{and}\quad\Lambda(\Om)=4\dfrac{\pi^2}{L^2}=4\lambda(\Om).\]
We notice that in this case 
\[v(x) =\dfrac{1}{2}\left(1-\cos\left(2\dfrac{\pi}{L}x\right)\right) =\sin^2\left(\dfrac{\pi}{L}x\right) = u^2(x).\]
The previous relation between the eigenfunctions of the one-dimensional Dirichlet and the Buckling eigenvalue problems is the main idea in the proof of inequality \eqref{Payneineq}.\medskip

Lastly, we point out that inequalities of the type \eqref{Payneineq} also hold true in the context of Riemannian manifolds under positivity assumptions of the Ricci tensor (see \cite{IliasShouman,HuangMa}). 
 
\section{Computations for an Infinite Strip}

To have an idea of why the factor $4$ in \eqref{Payneineq} was thought to be sharp and why it is not, let $n=2$ and let $H=\R\times(0,\pi)$ be an infinite strip with $\lambda(H)=1$, and consider the boundary value problem
\begin{equation}\label{stripbuckling}\begin{cases}
    -\Delta^2 w = \Lambda \Delta w &\text{in }H,\\[5pt]
    w=\partial_y w = 0 &\text{if }y=0,\pi.
\end{cases}\end{equation}
In analogy to what happens with the Laplacian, one may think that the smallest $\Lambda$ for which a non-trivial solution exists is achieved among functions that are constant with respect to the variable $x$, hence reducing the boundary value problem to the ODE 
\begin{equation}\label{auxiliaryeigprob1d0}\begin{cases}
    - w^{({\mathrm{\romannumeral 4}})} = \Lambda w'' &\text{in }(0,\pi),\\[5pt]
    w(0)=w(\pi)=w'(0)=w'(\pi) = 0 
\end{cases}\end{equation}
which admits non-trivial solutions if and only if $\Lambda>0$ is a solution to 
\begin{equation}\label{implicitlambda0}
    2\ml(\cos\ml(\sqrt\Lambda \pi\mr)-1\mr)+\sqrt\La \pi \sin\ml(\sqrt\La \pi\mr)=0.
\end{equation}
For such $\La$, the solutions to the equation \eqref{auxiliaryeigprob1d0} are proportional to
\[w(y)=-\ml(\sqrt\La \pi -\sin\ml(\sqrt\La \pi\mr)\mr)\ml(1-\cos\ml(\sqrt\La y\mr)\mr)+\ml(1-\cos\ml(\sqrt\La \pi\mr)\mr)\ml(\sqrt\La y -\sin\ml(\sqrt\La y\mr)\mr).\]
One can verify that the smallest positive solution to \eqref{implicitlambda0} is $\La=4$, hence, the smallest eigenvalue of the strip would indeed be $\La(H)=4=4\la(H)$. However, if we do not impose $u$ to be constant with respect to $x$ we can find different solutions with possibly smaller eigenvalues. For example, for every $\mu\ge0$, we can look for solutions of the type
\[w(x,y)=h(y)\cos(\sqrt{\mu}x)\] 
which exist if and only if $h$ is a solution to
\begin{equation}\label{auxiliaryeigprob1d}\begin{cases}
    - (h''-\mu h)'' = (\Lambda-\mu) (h''-\mu h) &\text{in }(0,\pi),\\[5pt]
    h(0)=h(\pi)=h'(0)=h'(\pi) = 0.
\end{cases}\end{equation}
The previous equation admits non-trivial solutions if and only if $\Lambda\ge\mu$ is a solution to
\begin{equation}\label{implicitlambda} 2\sqrt{\mu}\sqrt{\Lambda-\mu}\left(\cosh\left(\sqrt{\mu}\pi\right)\cos\left(\sqrt{\Lambda-\mu}\pi\right)-1\right)+(\Lambda-2\mu)\sinh\left(\sqrt{\mu}\pi\right)\sin\left(\sqrt{\Lambda-\mu}\pi\right)=0.
\end{equation}
For such $\La$, the solutions to equation \eqref{auxiliaryeigprob1d} are proportional to
\[h(y)=-\ddfr{\cosh\left(\sqrt{\mu}y\right)-\cos\left(\sqrt{\Lambda-\mu}\,y\right)}{\cosh\left(\sqrt{\mu}\pi\right)-\cos\left(\sqrt{\Lambda-\mu}\,\pi\right)}+\ddfr{\sqrt{\Lambda-\mu}\sinh\left(\sqrt{\mu}y\right)-\sqrt{\mu}\sin\left(\sqrt{\Lambda-\mu}\,y\right)}{\sqrt{\Lambda-\mu}\sinh\left(\sqrt{\mu}\pi\right)-\sqrt{\mu}\sin\left(\sqrt{\Lambda-\mu}\,\pi\right)}.
\]
The same ODE was also studied in \cite{BuosoPArini} to better understand the Buckling eigenvalue of an annulus in the limit when the difference in radii goes to zero. For every $\mu\ge0$, let $\Lambda_\mu$ be the smallest solution to \eqref{implicitlambda}.
From a variational point of view, we have
\begin{equation}\label{Stripeig}\Lambda_\mu = \min_{h\in H^2_0(0,\pi)}\Set{\mu+\ddfr{\int_0^\pi (h'')^2\,dy+\mu\int_0^\pi (h')^2\,dy}{\int_0^\pi (h')^2\,dy+\mu\int_0^\pi h^2\,dy}}.\end{equation}
Which can be estimated from above using $h(y)=\sin^2(y)$ as a test function, obtaining
\[\Lambda_\mu\le \mu + \dfrac{16+4\mu}{4+3\mu},\quad
\text{and}\quad
\min_{\mu\ge0} \Lambda_\mu \le \dfrac{8}{3}\sqrt{2}.\]
Numerically, the minimum with respect to $\mu$ of $\Lambda_\mu$ can be estimated to be 
\[\min_{\mu\ge0}\Lambda_\mu\approx3.7570\dots\]
Hence, the couple $w(y)=\sin^2(y)$ and $\Lambda=4$ is a solution to the boundary value problem \eqref{stripbuckling}, but $4$ is not the smallest value of $\Lambda$ for which a non-trivial solution exists. Namely, if we define the buckling eigenvalue of an unbounded set as the infimum of the eigenvalues of its bounded open subsets, we can formalize the previous observations to have the following
\begin{teor}
    Let $L>0$, and let $H=\R^{n-1}\times(0,L)$ be an infinite strip in $\R^n$, then
    \[\Lambda(H) \le \sigma \lambda(H),\]
    where 
    \[\sigma := \min_{\mu\ge0} \Lambda_{\mu}<4.\]
\end{teor}
\begin{proof}
Starting from the solutions to \eqref{auxiliaryeigprob1d}, we can construct appropriate test functions on a monotone sequence $\Om_k$ of parallelepipeds to show that for every $\mu>0$
\[\lim_{k} \dfrac{\Lambda(\Om_k)}{\lambda(\Om_k)}\le \Lambda_\mu,\]
where $\Lambda_\mu$ is  defined in \eqref{Stripeig}.
We will show the result for $n=2$, as it is analogous for $n\ge3$. Since the ratio 
\[\dfrac{\Lambda(\Om)}{\lambda(\Om)}\]
is scaling invariant, we can assume $L=\pi$ so that $\lambda(H)=1$. For every $\mu>0$ let $h$ be a minimizer to \eqref{Stripeig} with 
\[\int_0^\pi\ml(h'\mr)^2\,dy+\mu\int_0^\pi h^2\,dy=\dfrac{2\pi}{\sqrt{\mu}}.\] Let $L_\mu=\pi/\sqrt{\mu}$ and $R_1=(0,L_\mu)\times(0,\pi)$, and notice that the function $w_1(x,y)=h(y)\cos(\sqrt{\mu}x)$ solves
\[
\begin{dcases}
  -\Delta^2 w_1 = \Lambda_\mu \Delta w_1 &\text{in }R_1,\\[5pt]
  w_1=\partial_y w_1 = 0 &\text{if }y=0,\pi, \\ 
  \partial_x w_1 =\partial_x \De w_1= 0 &\text{if }x=0,L_\mu.
\end{dcases}
\]
In particular, 
\[
  \int_{R_1} (\De w_1)^2\,dx =\La_\mu \int_{R_1}\abs{\na w_1}^2\,dx
\]
and, by construction,
\[
\int_{R_1}\abs{\na w_1}^2\,dx = 1.
\]
The idea is to extend $w_1$ periodically on the infinite strip $H$. Let $k$ be a natural number, and consider the sequence of rectangles 
\[
\Omega_k=(-kL_\mu,kL_\mu)\times(0,\pi).
\]
Let $\varphi_k$ be a sequence of smooth cut-off functions, equi-bounded in $C^2$, such that \[\varphi_k(x)=1\quad\text{ when }\quad \abs{x}\le(k-1)L_\mu,\] and with $\supp \varphi \subset (-kL_\mu,k L_\mu)$. We define $w_k(x,y)=h(y)\cos(\sqrt{\mu}x)\varphi_k(x)$, so that $w_k\in H^2_0(\Om_k)$. We notice that, by periodicity, for every $j$ between $-k+2$ and $k-1$ we have that on the rectangle $R_j=((j-1)L_\mu,jL_\mu)\times (0,\pi)$ it holds
\[
\int_{R_j} (\De w_k)^2\,dx = \int_{R_1}(\De w_1)^2\,dx = \La_\mu \int_{R_1}\abs{\na w_1}^2\,dx = \La_\mu \int_{R_j}\abs{\na w_k}^2\,dx,
\]
and in particular, 
\[
  \int_{\Om_k} (\De w_k)^2\,dx = 2(k-1)\La_\mu + Q_1
\] 
where $Q_1$ is a remainder uniformly bounded with respect to $k$. Analogously
\[
  \int_{\Om_k} \abs{\na w_k}^2\,dx = 2(k-1)+ Q_2
\]
where $Q_2$ is uniformly bounded as well. Hence, using $w_k$ as a test function for $\La(\Om_k)$, we get
\[\Lambda(\Om_k)\le\frac{\int_{\Om_k} \ml(\Delta w_k\mr)^2\,dx\,dy}{\int_{\Om_k}\abs{\nabla w_k}^2\,dx\,dy}\le \fr{2(k-1)\La_\mu +Q_1}{2(k-1)+Q_2}.\]
Then, passing to the limit as $k$ goes to infinity, and using the fact that $\Om_k$ is an increasing sequence converging to $H$, we have that
\[\lim_{k} \Lambda(\Om_k)\le \Lambda_\mu.\]
Hence, recalling that $\sigma=\min_{\mu\ge0} \Lambda_\mu$ and that $\la(H)=1$, we have
\[\Lambda(H)\le\sigma \lambda(H).\]
\end{proof}

\section{Proof of the main result}
In the following, we will use the improved $\log$-concavity estimate proved in \cite[Theorem 1.5]{AndClutt2011} for positive Dirichlet eigenfunctions 
 \begin{teor}[Improved $\log$-concavity]
  \label{thm: improvedLog}
Let $\Om\sbse \R^n$ be a bounded convex open set, let $D(\Om)$ be its diameter, and let $u$ be a positive first eigenfunction associated to $\la(\Om)$. If $v=\log(u)$, then for every $\eta\in\R^n$ and for every $x\in \Om$ we have
\[
  -\langle D^2v(x)\eta,\eta\rangle\ge \fr{\pi^2}{D(\Om)^2}\abs{\eta}^2 .
\]
 \end{teor}
 
\begin{prop}[Improved Payne inequality]\label{prop:improvedPayne}
Let $n\ge2$ and let $\Om\subset\R^n$ be a bounded convex set. Then
\begin{equation}\label{improvedPayne}\Lambda(\Omega)\le \left(4-2\mathcal{T}(\Om)\right)\lambda(\Omega),\end{equation}
where 
\[\mathcal{T}(\Om)=\dfrac{(n-1)\pi^2}{D(\Om)^2\lambda(\Om)}.\] 
\end{prop}
\begin{proof}
    Let $u\ge0$ be a Dirichlet eigenfunction associated to $\la(\Om)$. Following the Payne approach, we can use $u^2$ as a test function in the variational definition of $\Lambda(\Om)$. 
    \begin{equation}
      \label{eq: imprvPayne1}
\Lambda(\Om)\le\dfrac{\displaystyle\int_\Om \left(\Delta u^2\right)^2\,dx}{\displaystyle\int_\Om \abs*{\nabla u^2}^2\,dx}=\dfrac{\displaystyle\int_\Om \left(\lambda(\Om)^2u^4+\abs{\nabla u}^4-2\lambda(\Om)u^2\abs{\nabla u}^2\right)\,dx}{\displaystyle\int_\Om u^2\abs{\nabla u}^2\,dx}.
\end{equation}
Using the equation satisfied by $u$, we immediately have 
\begin{equation}
  \label{eq: imprvPayne2}\lambda(\Om)\int_\Om u^4\,dx=3\int_\Om u^2\abs{\nabla u}^2\,dx.\end{equation} 
While, letting
\[H_u=-\divv\left(\dfrac{\nabla u}{\abs{\nabla u}}\right)\]
the mean curvature of the level sets of the function $u$, we can write
\begin{equation}
  \label{eq: imprvPayne3}\abs{\nabla u}^4=\divv\left(u\abs{\nabla u}^2 \nabla u\right)+3\lambda(\Om) u^2\abs{\nabla u}^2-2u\abs{\nabla u}^3 H_u.\end{equation}
Hence, joining \eqref{eq: imprvPayne1} with \eqref{eq: imprvPayne2} and \eqref{eq: imprvPayne3},
\[\Lambda(\Om)\le4\lambda(\Om)-2\dfrac{\displaystyle\int_\Om u\abs{\nabla u}^3 H_u\,dx}{\displaystyle\int_\Om u^2\abs{\nabla u}^2\,dx}.\]
Finally, we claim that the improved log-concavity of $u$ (\autoref{thm: improvedLog}) allows us to estimate 
\begin{equation}\label{curvaturebound}H_u\ge \dfrac{(n-1) \pi^2}{D(\Om)^2} \dfrac{u}{\abs{\nabla u}},\end{equation}
which ensures equation \eqref{improvedPayne}. Indeed, at every $x\in\Om$ with $\nabla u(x) \ne 0$, by definition of $H_u$, we can rewrite it as
\begin{equation}\label{meancurv}H_u=-\dfrac{\displaystyle\sum_{i=1}^{n-1} \langle D^2 u(x) \tau_i,\tau_i\rangle}{\abs{\nabla u(x)}},\end{equation}
where $\set{\tau_1,\tau_2,\dots,\tau_{n-1}, \nabla u(x)/\abs{\nabla u(x)}}$ is an orthonormal base of $\R^n$. 
Then, using the improved log-concavity of $u$, we have that for every $\eta\in\R^n$ 
\[-\langle D^2(\log(u)) \eta,\eta\rangle \ge \dfrac{\pi^2}{D(\Om)^2}\abs{\eta}^2,\]
so that, if $\eta$ is orthogonal to $\nabla u(x)$, we have,
\[-\langle D^2 u(x) \eta,\eta\rangle \ge \dfrac{\pi^2}{D(\Om)^2}u(x)\abs{\eta}^2,\]
which, together with \eqref{meancurv} implies estimate \eqref{curvaturebound}.
\end{proof}

\begin{defi}
    Let $\Om\sbs\R^n$ be a convex set. We define the \emph{support function} of $\Om$ as follows: for every $\nu\in\Spn$,
    \[
    h_\Omega(\nu):=\max_{x\in\Om} (x\cdot\nu).
    \]
Then, we define the \emph{minimal width} of $\Om$ as
\[
w(\Omega)=\min_{\nu\in\Spn} \left(h_{\Om}(\nu)+h_{\Om}(-\nu)\right).
\]
\end{defi}

\begin{oss}
The symmetric support function $h_\Om(\nu)+h_\Om(-\nu)$ measures the distance between two supporting hyperplanes of $\Om$ and orthogonal to the directions $\nu$ and $-\nu$. The minimal width is then, by definition, the minimum of such distances, while the diameter gives the maximum. Hence, the ratio of the minimal width and the diameter of a convex set is a commonly used measure to quantify the "thinness" of a convex set.
\end{oss}
\begin{oss}
Let us notice that the correction term 
\[\mathcal{T}(\Om)=\dfrac{(n-1)\pi^2}{D(\Om)^2\lambda(\Om)},\] 
in equation \eqref{improvedPayne} is a measure of the thinness of the set $\Om$. Indeed, let us recall that for convex sets $\Om$ there exists a dimensional constant $c=c_n$ such that
\begin{equation}
\label{eq: equivlambdawidth}\dfrac{1}{c}\le w(\Om)^{2}\lambda(\Om)\le c.\end{equation}
For the upper bound it is sufficient to consider  an infinite strip $H$ of width $w(\Om)$ containing $\Om$ and, by monotonicity,
\[\la(\Om)\ge\la(H)=\dfrac{\pi^2}{w(\Om)^2}.\]
For the upper bound, we consider $r_\Om$
the inradius of $\Om$, and by~\cite[\S 10.44 Formula (9)]{BonnFench1987} we know that there exists $c_n>0$ such that,
\begin{equation}
\label{widthinradius}
  w(\Om)\le c_n r_\Om.
\end{equation} 
Then, if $B_{r_\Om}$ is a ball inscribed in $\Om$, by monotonicity, we have that
\[\la(\Om)\ge \la(B_{r_\Om})= \dfrac{c_n}{r_\Om^2},\] 
which together with \eqref{widthinradius} proves the upper bound.

Using~\eqref{eq: equivlambdawidth} we estimate $\mathcal{T}(\Om)$ as 
\[\dfrac{1}{c} \left(\dfrac{w(\Om)}{D(\Om)}\right)^2\le \mathcal{T}(\Om)\le c \left(\dfrac{w(\Om)}{D(\Om)}\right)^2.\]
Therefore, the factor in front of $\la(\Om)$ in~\eqref{improvedPayne} gets closer to $4$ as the set $\Om$ becomes thinner, meaning that the inequality~\eqref{improvedPayne} is a ``good'' improvement of Payne's inequality only for non-thin sets. In the following, we will prove a different inequality, which will be ``good'' for thin sets.

\end{oss}

\begin{lemma}\label{lemma}
  Let $n\ge2$, let $\Om\subset\R^n$ be a bounded convex set, and let $u$ be the first Dirichlet eigenfunction on $\Om$. Then, for every $\nu\in\mathbb{S}^{n-1}$, the unit sphere in $\R^n$, and for every $\alpha>0$, we have
  
  \begin{equation}\label{estimatenu}\Lambda(\Om)\le\left(\dfrac{16-8\mathcal{T}(\Om)+8\alpha+3\alpha^2}{4+3\alpha}+\dfrac{48\alpha}{4+3\alpha}\ddfr{\int_\Om u^2(\nabla u\cdot\nu)^2\,dx}{\lambda(\Om)\int_\Om u^4\,dx}\right)\lambda(\Om)\end{equation}
\end{lemma}
\begin{proof}
  \proofpart{Estimating with oscillating test functions}
  We claim that for every $\mu>0$ and for every $h\in H^2_0(\Om)$ we have 
  \begin{equation}
    \label{eq: oscillatingTest}
    \La(\Om)\le \mu+ \ddfr{\int_{\Om}^{}(\De h)^2\,dx+\mu \int_\Om \abs{\na h}^2\,dx}{\int_\Om \abs{\na h}^2\,dx+\mu \int_{\Om}^{}h^2\,dx }+ 4\mu \ddfr{\int_{\Om}^{}(\nabla h \cdot \nu)^2\,dx}{\int_\Om \abs{\na h}^2\,dx+\mu \int_{\Om}^{}h^2\,dx}.
  \end{equation}
  Let $h\in H^2_0(\Om)$ and for every positive $\mu$ let 
  \[
    s_\mu(x) = \sin(\sqrt{\mu}\,x\cdot\nu), \qquad \qquad    c_\mu(x) = \cos(\sqrt{\mu}\,x\cdot\nu).
  \]
  We define $\vp_c(x) = h(x)c_\mu(x)$ and \(\vp_s(x)=h(x)s_\mu(x)\). By definition of $\La$ we have that
  \[
    \La(\Om) \le 
    \min \ml\{\ddfr{\int_{\Om}^{}(\De \varphi_c)^2\,dx}{\int_{\Om}^{}\abs{\na \varphi_c}^2\,dx},\ddfr{\int_{\Om}^{}(\De \varphi_s)^2\,dx}{\int_{\Om}^{}\abs{\na \varphi_s}^2\,dx}\mr\}\le 
    \ddfr{\int_{\Om}^{}\ml((\De \varphi_c)^2+(\De \varphi_s)^2\mr)\,dx}{\int_{\Om}^{}\ml(\abs{\na \varphi_c}^2+\abs{\na \varphi_s}^2\mr)\,dx}.
  \]
  Straightforward computations give \eqref{eq: oscillatingTest}.
  \proofpart{Choosing $h$ as the square of the first Dirichlet eigenfunction}
  Let $u$ be the first Dirichlet Laplacian eigenfunction associated to $\la(\Om)$. If we take $h(x)=u^2$ then, following the computations in \autoref{prop:improvedPayne} we have the following two estimates:\[\int_\Om \abs*{\nabla u^2}^2\,dx=\dfrac{4\lambda(\Om)}{3}\int_\Om u^4\,dx,\] 
    \[
    \int_{\Om}^{}\left(\De u^2\right)^2\,dx\le \ml(4-2\mathcal{T}(\Om)\mr)\lambda(\Om)\int_{\Om}^{}\abs*{\na u^2}^2\,dx.
    \]
  In particular,
  \[\La(\Om)\le \mu+\ddfr{8(2-\mathcal{T}(\Om))\lambda(\Om)^2+4\mu\lambda(\Om)}{4\lambda(\Om)+3\mu}+\dfrac{48\mu}{4\lambda(\Om)+3\mu}\ddfr{\int_\Om u^2(\nabla u\cdot\nu)^2\,dx}{\lambda(\Om)\int_\Om u^4\,dx}.\]
  Finally, the assertion follows choosing $\mu=\alpha\lambda(\Om)$.
\end{proof}
From the previous lemma, we have the following proposition.
\begin{prop}[Improved payne inequality for thin sets]
  \label{prop: improvedPayneThin}
    Let $n\ge2$, then for every bounded convex set $\Om\subset\R^n$ we have \begin{equation}\label{thinPayne}
    \Lambda(\Om)\le \dfrac{8}{3}\sqrt{2}\left(1+\dfrac{2^{5/6}}{(n-1)^{1/3}} \mathcal{T}(\Om)^{1/3}\right)^3\lambda(\Om)
    \end{equation}
\end{prop}
\begin{proof}
    We start by considering the case of a cylinder of the type $C_{A,l}=A\times(0,l)$. Then 
    \[\lambda(C_{A,l})={\tilde \lambda}(A)+\dfrac{\pi^2}{l^2}\]
    and the first Dirichlet Laplacian eigenfunction is given by
    \[u(x',x_n)=u_A(x')\sin\left(\fr{\pi}{l} x_n\right),\]
    where $u_A$ is the first Dirichlet Laplacian eigenfunction on $A$. Then, for every $\alpha\ge0$, by \eqref{estimatenu}, with $\nu=\mathbf{e}_n$, we have
   \[\begin{split}
   \Lambda(C_{A,l})&\le \dfrac{16+8\alpha+3\alpha^2}{4+3\alpha}\lambda(C_{A,l})+\dfrac{48\alpha}{4+3\alpha}\ddfr{\int_{C_{A,l}} u^2(u_{x_n})^2\,dx}{\int_{C_{A,l}} u^4\,dx}\\[10 pt]
   &\le \dfrac{16+8\alpha+3\alpha^2}{4+3\alpha}\lambda(C_{A,l}) + \dfrac{16\alpha}{4+3\alpha}\dfrac{\pi^2}{l^2}\\[10 pt]
   &\le\dfrac{16+8\alpha+3\alpha^2}{4+3\alpha}{\tilde \lambda}(A) + \dfrac{16+24\alpha+3\alpha^2}{4+3\alpha}\dfrac{\pi^2}{l^2}.
   \end{split}\]
   Minimizing with respect to $\alpha\ge0$ the coefficient of ${\tilde \lambda}(A)$, we have 
   \begin{equation}\label{estimatescilinders}
   \La(C_{A,l})\le \dfrac{8}{3}\sqrt{2}\left({\tilde \lambda}(A)+\sqrt{2}\dfrac{\pi^2}{l^2}\right)
   \end{equation}
   Let now $\Om$ be a bounded convex set, the idea is to estimate the buckling eigenvalue of $\Om$ with the one of an appropriate cylinder. Without loss of generality, we can assume that a diameter of $\Om$ lies on the $x_n$ axis. For every $t$ let \[\Om'_t=\set{x'\in\R^{n-1}|\,(x',t)\in\Om}.\] 
   
    For every $t$ such that $(0',t)\in\Omega$ we have that either $(0',t- D(\Om)/2)\in\Om$ or $(0',t+ D(\Om)/2)\in\Om$, and without loss of generality we assume the latter. By the convexity of $\Om$, we have that the convex hull of $\ml(\Om'_{t}\times\set{t}\mr)\cup\set{(0',t+ D(\Om)/2)}$ is contained in $\Om$ (see \autoref{fig:CylinderInsideOm}). Hence, for every $z\in(0,1)$ the set 
    \[C_z=(1-z)\Om'_{t}\times(t,t+z D(\Om)/2)\subset\Om.\]
    
    \begin{figure} \begin{tikzpicture}
        \node[anchor=south west,inner sep=0] (image) at (0,0){\includegraphics[width=.95\linewidth]{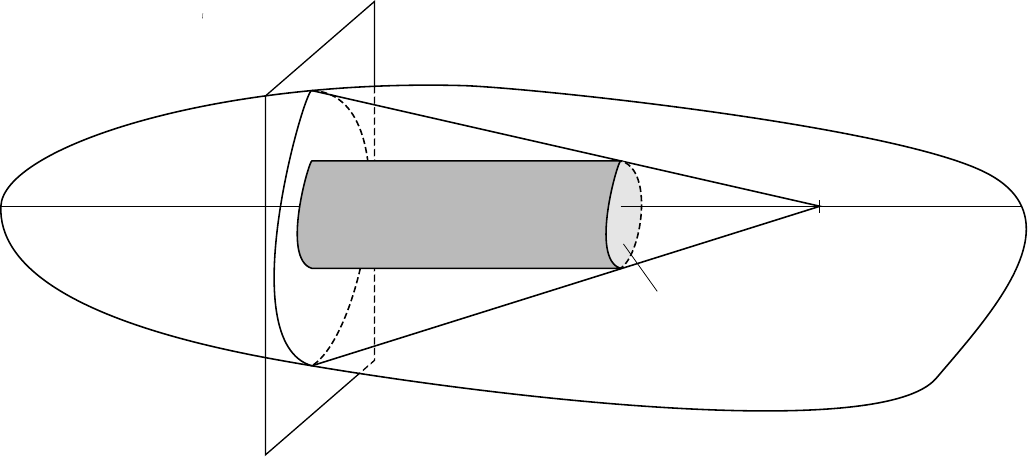}};
        \begin{scope}[x={(image.south east)},y={(image.north west)}]
        \node[anchor=south west] at (.5,.5) {$C_z$};
        \node[anchor=south west] at (.6,.3) {$(1-z)\Om_{t}'$};
        \node[anchor=west] at (.31,.1) {$x_n=t$};
        \node[anchor=south west] at (.9,.1) {$\Om$};
        \node[anchor=west] at (.8,.5) {$t+\frac{D}{2}$};
        \node[] at (.3,.3) {$\Om'_{t}$};
        \end{scope}
    \end{tikzpicture}
      \caption{Construction of the cylinder $C_{z}$}
      \label{fig:CylinderInsideOm}
    \end{figure}
    By domain monotonicity and estimate \eqref{estimatescilinders}, we have
    \[\Lambda(\Om)\le\Lambda(C_z)\le \dfrac{8}{3}\sqrt{2}\left(\dfrac{\tilde \lambda(\Om'_{t})}{(1-z)^2}+4\sqrt{2}\dfrac{\pi^2}{z^2D^2(\Om)}\right).\]
     
  From \cite[Proposition 2.3]{Beck2020} we know that
   \[\inf_{t}\tilde\lambda(\Om'_t)\le\lambda(\Om),\]
    so that, passing to the infimum in $t$, we have
    \[
    \Lambda(\Om)\le\dfrac{8}{3}\sqrt{2}\lambda(\Om)\left(\dfrac{1}{(1-z)^2}+\dfrac{4\sqrt{2}}{z^2(n-1)}\mathcal{T}(\Om)\right).
 \]
    Finally, we obtain the assertion passing to the minimum in $z\in(0,1)$.
\end{proof}

    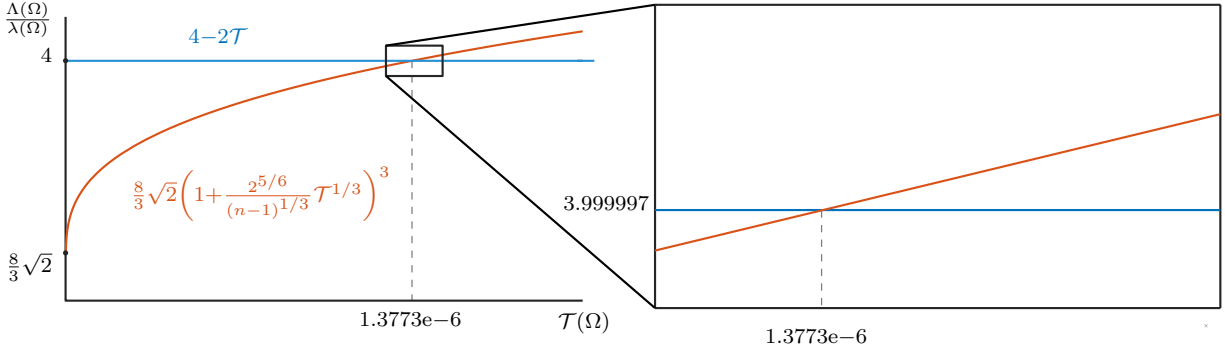
\begin{figure} \begin{tikzpicture}
        \node[anchor=south west,inner sep=0] (image) at (0,0){\includegraphics[width=.9\linewidth]{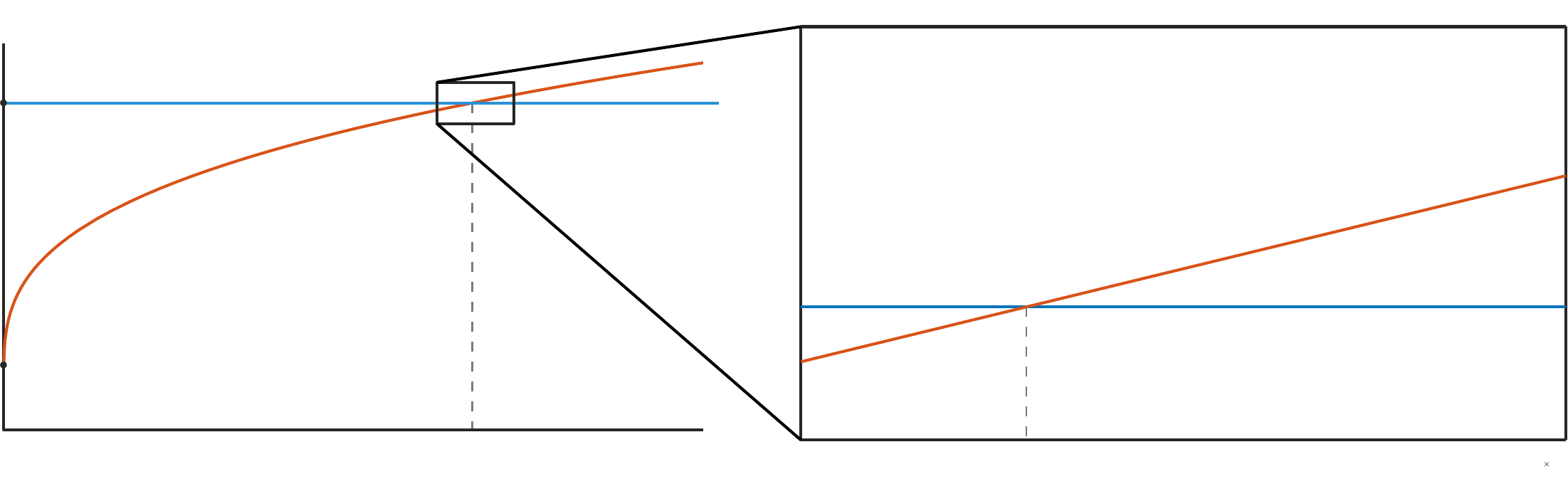}};
        \begin{scope}[x={(image.south east)},y={(image.north west)}]
        \node[anchor=east] at (0,.9) {$\scriptstyle\frac{\La(\Omega)}{\la(\Om)}$};
        \node[anchor=east] at (0,.8) {$\scriptstyle4$};
        \node[anchor=east] at (0,.2) {$\scriptstyle \frac{8}{3}\sqrt{2}$};
        \node[anchor=north] at (.45,.1) {$\scriptstyle\mathcal{T}(\Om)$};
        \node[anchor=east] at (.515,.38) {$\scriptstyle3.999997$};
        \node[anchor=north] at (.65,.05) {$\scriptstyle1.3773\mathrm{e}-6$};
        \node[anchor=north] at (.3,.1) {$\scriptstyle1.3773\mathrm{e}-6$};
        \definecolor{graphBlue}{HTML}{0072bd}
        \definecolor{graphRed}{HTML}{d95319}
        \node[anchor=south west] at (.05 ,.3) {\textcolor{graphRed}{$\scriptstyle{\frac{8}{3}\sqrt{2}\left(1+\frac{2^{5/6}}{(n-1)^{1/3}} \mathcal{T}^{1/3}\right)^3}$}};
        \node[anchor=south west] at (.1,.8) {\textcolor{graphBlue}{$\scriptstyle 4-2\mathcal{T}$}};
        \end{scope}
    \end{tikzpicture}
      \caption{Graph of the two inequalities in \autoref{prop:improvedPayne} (in blue) and \autoref{prop: improvedPayneThin} (in orange) in the case $n=2$.}
      \label{fig:graphInequalities}
    \end{figure}
We can now prove the main result
\begin{proof}[Proof of \autoref{main}]
    Let $n\ge2$ and let $\Om_k\subset\R^n$ be a maximizing sequence of convex sets for  \[\sup_{\Om \text{ convex}}\dfrac{\Lambda(\Om)}{\lambda(\Om)}.\]
    Then, either there exists $C>0$ such that other $\mathcal{T}(\Om_k)\ge c$ or
    \[\lim_{k}\mathcal{T}(\Om_k)=0.\]
   If $\mathcal{T}(\Om_k)\ge c$, using \eqref{improvedPayne}, we have
    \[\sup_{\Om \text{ convex}}\dfrac{\Lambda(\Om)}{\lambda(\Om)}=\lim_k \dfrac{\Lambda(\Om_k)}{\lambda(\Om_k)}\le 4-2c<4.\]
    If 
   \[\lim_{k}\mathcal{T}(\Om_k)=0,\]
   using \eqref{thinPayne}, we have
   \[\sup_{\Om \text{ convex}}\dfrac{\Lambda(\Om)}{\lambda(\Om)}=\lim_k \dfrac{\Lambda(\Om_k)}{\lambda(\Om_k)}\le \dfrac{8}{3}\sqrt{2}<4.\]
\end{proof}

 \begin{oss}
For higher dimensions, we can use estimate \eqref{estimatenu} to obtain an explicit estimate of the constant $C_n$ in \autoref{main}. In particular, if $n\ge5$ then for every bounded convex set $\Om\subset \R^n$ we have 
\begin{equation}\label{largedimPayne}\Lambda(\Om) \le \dfrac{8}{3}\left(\sqrt{2-\dfrac{4}{n}}+\dfrac{2}{n}\right)\lambda(\Om).\end{equation}
Indeed, let $u$ be a first Dirichlet eigenfunction on $\Om$, then 
\[\inf_{\nu\in\mathbb{S}^{n-1}} \int_\Om u^2(\nabla u\cdot \nu)^2\,dx \le \dfrac{1}{n}\int_\Om u^2\abs{\nabla u}^2\,dx=\dfrac{\lambda(\Om)}{3n}\int_\Om u^4\,dx.\]
Therefore, using \eqref{estimatenu}, we have that for every $\alpha>0$
\[\Lambda(\Om)\le \left(\dfrac{16+8\left(1+\dfrac{2}{n}\right)\alpha+3\alpha^2}{4+3\alpha}-\dfrac{8\mathcal{T}(\Omega)}{4+3\alpha}\right)\lambda(\Om).\]
Minimizing the function
\[\alpha\in[0,+\infty)\mapsto\dfrac{16+8\left(1+\dfrac{2}{n}\right)\alpha+3\alpha^2}{4+3\alpha},\]
for $n=2,3,4$ we have once again the estimate \eqref{improvedPayne}, while for $n\ge5$ we have 
\[\begin{split}\Lambda(\Om) &\le \left(\dfrac{8}{3}\left(\sqrt{2-\dfrac{4}{n}}+\dfrac{2}{n}\right)-\sqrt{\dfrac{2n}{n-2}}\mathcal{T}(\Om)\right)\lambda(\Om)\\[10 pt] &\le \dfrac{8}{3}\left(\sqrt{2-\dfrac{4}{n}}+\dfrac{2}{n}\right)\lambda(\Om)
\end{split},\]
that is \eqref{largedimPayne}.
\end{oss}

\begin{open}
    Compute 
    \[
         \sup_{\Om \text{ convex}}\dfrac{\Lambda(\Om)}{\lambda(\Om)}.
    \]
    Is it true that the supremum is achieved by the infinite strip?
 
\end{open}

\subsubsection*{Acknowledgements} 
The four authors are members of Gruppo Nazionale per l’Analisi Matematica, la Probabilità e le loro Applicazioni
(GNAMPA) of Istituto Nazionale di Alta Matematica (INdAM). 

 The author Cristina Trombetti has been supported by the Project MUR PRIN-PNRR 2022: "Linear and Nonlinear PDE’S: New directions and Applications", P2022YFAJH

The authors Paolo Acampora and Emanuele Cristoforoni were partially supported by the INdAM - GNAMPA Project, 2024: "Modelli PDE-ODE nonlineari e proprieta' di PDE su domini standard e non-standard", CUP\_E53C23001670001. 

The author Paolo Acampora was partially supported by the INdAM - GNAMPA Project, 2025,
”Analisi di Problemi Inversi nelle Equazioni alle Derivate Parziali”, CUP\_E5324001950001

The author Emanuele Cristoforoni was partially supported by the INdAM - GNAMPA Project, 2025,
”Esistenza, unicità, simmetria e stabilità per problemi ellittici nonlineari e nonlocali”, CUP\_E5324001950001

The authors Paolo Acampora and Carlo Nitsch is partially supported by Centro Nazionale HPC, Big Data e Quantum Computing, (CN\_00000013)(CUP: E63C22000980007), under the PNRR MUR program funded by the NextGenerationEU.

\printbibliography[heading=bibintoc]
\Addresses
\end{document}